\documentclass{article}

\usepackage[latin1]{inputenc}
\usepackage[T1]{fontenc}

\usepackage{subfigure}

\usepackage{amsmath}
\usepackage{amsfonts}
\usepackage{amssymb}
\usepackage{amsthm}
\usepackage{graphicx}
\usepackage{enumerate}
\usepackage{bbm,stmaryrd}

\newtheorem{theorem}{Theorem}

\newtheorem{corollary}[theorem]{Corollary}

\newtheorem{lemma}[theorem]{Lemma}

\newtheorem{proposition}[theorem]{Proposition}
\newtheorem{remark}[theorem]{Remark}

\numberwithin{equation}{section}

\newcommand{\Z}{\mathbb{Z}}

\renewcommand{\P}{\mathbb{P}}
\newcommand{\E}{\mathbb{E}}

\newcommand{\C}{\mathcal{C}}

\newcommand{\be}{\beta}
\newcommand{\g}{\gamma}
\newcommand{\D}{\mathcal{D}}
\renewcommand{\H}{\mathbb{H}}
\newcommand{\e}{\epsilon}
\newcommand{\B}{\mathcal{B}}

\newcommand{\A}{\mathcal{A}}
\newcommand{\leqst}{\leq_{\textrm{st}}}
\newcommand{\Hw}{\mathbf{HM}_{\circ}}
\newcommand{\Hb}{\mathbf{HM}_{\bullet}}

\title{Connection probabilities and RSW-type bounds for the FK Ising model}
\author{Hugo Duminil-Copin, Cl\'ement Hongler, Pierre Nolin}
\date{}
\begin{document}

\maketitle

\begin{abstract} We prove Russo-Seymour-Welsh-type uniform bounds on crossing probabilities for the FK Ising model at criticality, independent of the boundary conditions. Our proof relies mainly on Smirnov's fermionic observable for the FK Ising model, which allows us to get precise estimates on boundary connection probabilities. It remains purely discrete, in particular we do not make use of any continuum limit, and it can be used to derive directly several noteworthy results -- some new and some not -- among which the fact that there is no spontaneous magnetization at criticality, tightness properties for the interfaces, and the existence of several critical exponents, in particular the half-plane one-arm exponent.
\end{abstract}

\section{Introduction}

It is fair to say that the two-dimensional Ising model has a very particular historical importance in statistical mechanics. This model of ferromagnetism has been the first natural model where the existence of a phase transition, a property common to many statistical mechanics models, has been proved, in Peierls' 1936 work \cite{Peierls}. In a series of seminal papers (particularly \cite{Onsager}), Onsager computed several macroscopic quantities associated with this model. Since then, the Ising model has attracted a lot of attention, and it has probably been one of the most studied models, giving birth to an extensive literature, both mathematical and physical.

A few decades later, in 1969, Fortuin and Kasteleyn introduced a dependent percolation model, for which the probability of a configuration is weighted by the number of clusters (connected components) that it contains. This percolation representation turned out to be extremely powerful to study the Ising model, and by now it has become known as the \emph{random-cluster} model, or the \emph{Fortuin-Kasteleyn percolation} -- \emph{FK percolation} for short. Recall that on a finite graph $G$, the FK percolation process with parameters $p,q$ is obtained by assigning to each configuration $\omega$ a probability proportional to
$$p^{o(\omega)} (1-p)^{c(\omega)} q^{k(\omega)},$$
where $o(\omega)$, $c(\omega)$, and $k(\omega)$ denote respectively the number of open edges, closed edges, and connected components in $\omega$. The definition of the model also involves the use of \emph{boundary conditions}, encoding connections taking place outside $G$. The boundary conditions can be seen as a set of additionnal edges between sites on the outer boundary, and they will play a central role in this article. The precise setup that we consider in this paper is presented in Section 2.

For the specific value $q=2$, the FK percolation provides a \emph{geometric representation} of the Ising model \emph{via} the \emph{Edwards-Sokal coupling} \cite{edwards-sokal}. For this reason, we restrict ourselves here to this value $q=2$, and we call this model the \emph{FK Ising model}. We also stick to the square lattice $\mathbb{Z}^2$ -- or subgraphs of it -- though our arguments could possibly be carried out in the more general context of isoradial graphs, as in \cite{CS2}. Note that our results are stated for the FK representation, but that the Edwards-Sokal coupling then allows one to translate them into results for the Ising model itself. For instance, 2-point connection probabilities for the FK Ising model correspond via this coupling to 2-spin correlation functions for the Ising model.

For the value $q=2$ and $\mathbb{Z}^2$ as an underlying graph, the model features a phase transition -- in the infinite-volume limit -- at the critical and self-dual point $p_c =  p_{\mathrm{sd}} = \frac{\sqrt{2}}{1+\sqrt{2}}$: for $p < p_c$, there is a.s. no infinite open cluster, while for $p>p_c$, there is a.s. a unique one. These two regimes, known as sub-critical and super-critical, have totally different macroscopic behaviors. Between them lies a very interesting and rich regime, the critical regime, corresponding to the value $p=p_c$. Its behavior is intimately related to the behavior of the model through its phase transition, as indicated in particular by the scaling theory.

In this paper, we prove lower and upper bounds for crossing probabilities in rectangles of bounded aspect ratio. These bounds are uniform in the size of the rectangles and in the boundary conditions, and they are analogues for the FK Ising model to the celebrated Russo-Seymour-Welsh bounds for percolation \cite{RSW1, RSW2}. Formally, we consider \emph{rectangles} $R$ on the square lattice, $\llbracket0,n\rrbracket\times\llbracket0,m\rrbracket$ for $n, m >0$, and translations of it -- here and in the following, $\llbracket\cdot,\cdot\rrbracket$ denotes the integer interval between the two (real) end-points, \emph{i.e.} the interval $[\cdot,\cdot] \cap \mathbb{Z}^2$. We denote by $\C_v(R)$ the event that there exists a \emph{vertical crossing} in $R$, a path from the bottom side $\llbracket0,n\rrbracket \times \{0\}$ to the top side $\llbracket0,n\rrbracket \times \{m\}$ which consists only of open edges. Our main result is the following:

\begin{theorem}[RSW-type crossing bounds] \label{RSW}
Let $0 < \be_1 < \be_2$. There exist two constants $0 <c_- \leq c_+<1$ (depending only on $\be_1$ and $\be_2$) such that for any rectangle $R$ with side lengths $n$ and $m \in \llbracket \be_1 n, \be_2 n\rrbracket$ (\emph{i.e.} with aspect ratio bounded away from $0$ and $\infty$ by $\be_1$ and $\be_2$), one has
$$c_- \leq \P_{ p_{\mathrm{sd}},2,R}^{\xi}(\C_v(R))\leq c_+$$
for \emph{any} boundary conditions $\xi$, where $\P_{ p_{\mathrm{sd}},2,R}^{\xi}$ denotes the FK measure on $R$ with parameters $(p,q)=( p_{\mathrm{sd}},2)$ and boundary conditions $\xi$.
\end{theorem}

These bounds are in some sense a first glimpse of scale invariance. It was widely believed in the physics literature that the FK Ising model at criticality, \emph{i.e.} for $p=p_c$, should possess a strong property of conformal invariance in the scaling limit \cite{P, BPZ1, BPZ2}. A precise mathematical meaning was recently established by Smirnov in a groundbreaking paper \cite{Sm3}. One of the main tools there is the so-called \emph{preholomorphic fermionic observable}, a complex observable that allows one to make appear harmonicity on the discrete level. This property can then be used to take continuum limits and describe the scaling limits so-obtained.

Our proof mostly relies on Smirnov's observable. More precisely, it is based on precise estimates on connection probabilities for boundary vertices, that allow us to use a second-moment method on the number of pairs of connected sites. For that, we use Smirnov's observable to reveal some harmonicity on the discrete level, which enables us to express macroscopic quantities such as connection probabilities in terms of discrete harmonic measures. Note in addition that other recent works \cite{BeffaraDuminilCopinSmirnov, DuminilExponentialDecay} also suggest that this complex observable is a relevant way to look at FK percolation, both for $q=2$ and for other values of $q$. We would like to stress that our argument is intended to be self-contained and that it stays completely in a discrete setting, using essentially elementary combinatorial tools: in particular, we do not make use of any continuum limits \cite{Sm4}.

\medbreak

Crossing bounds turned out to be instrumental to study the percolation model at and near its phase transition -- for instance to derive the \emph{scaling relations} \cite{Ke4}, that link the main macroscopic observables, such as the density of the infinite cluster and the characteristic length. These bounds are also useful to study variations of percolation, in particular for models exhibiting a self-organized critical behavior. We thus expect Theorem \ref{RSW} to be of particular interest to study the FK Ising model at and near criticality.

This theorem allows us to derive easily several noteworthy results. Among the consequences that we state, let us mention the celebrated fact that there is no magnetization at criticality for the Ising model, first established by Onsager in \cite{Onsager}, tightness results for the interfaces coming from the Aizenman-Burchard technology, and the value $1/2$ of the one-arm half-plane exponent -- that describes the asymptotic probability of large-distance connections starting from a boundary point, and also the decay of boundary magnetization in the Ising model. It should also be instrumental to prove the existence of critical exponents, in particular of the \emph{arm exponents}.

Theorem \ref{RSW} appears to be a very useful tool, enabling to transfer properties of the scaling limit objects back to the discrete models. Connections between discrete models and their continuum counterparts usually involve decorrelation of different scales, and thus use spatial independence between regions which are far enough from each other. In the random cluster model, one usually addresses the lack of spatial independence by successive conditionings, using repeatedly the spatial (or domain) Markov property of FK percolation, by which what happens outside a given domain can be encoded by appropriate boundary conditions. For this reason, proving bounds that are \emph{uniform in the boundary conditions} seems to be very important.

\medbreak
We would also like to mention that other proofs of Russo-Seymour-Welsh-type bounds have already been proposed. In \cite{CS2}, Chelkak and Smirnov give a direct and elegant argument to explicitly compute the crossing probabilities in the scaling limit, but their argument only applies for some specific boundary conditions (alternatively wired and free on the four sides). In \cite{CN3}, Camia and Newman also propose to obtain RSW as a corollary of a recently announced result: the convergence of the full collection of interfaces for the Ising model \cite{CS2} to the conformal loop ensemble CLE(3). The interpretation of CLE(3) in terms of the Brownian loop soup \cite{W3} is also used. However, to the author's knowledge, the proofs of these two results are quite involved, and moreover, the reasoning proposed only applies for boundary conditions ``in the bulk'', that correspond to the infinite-volume limit. In these two cases, uniformity with respect to the boundary conditions is not addressed, and there does not seem to be an easy argument to avoid this difficulty. While weaker forms might be sufficient for some applications, it seems however that this stronger form is needed in many important cases, and that it considerably shortens several existing arguments. 

\medbreak
The paper is organized as follows. In Section 2, we first remind the reader of the basic features of the FK percolation, as well as properties of Smirnov's fermionic observable. In Section 3, we compare the observable to harmonic measures, and we establish some estimates on these harmonic measures. These estimates are instrumental in the proof of Theorem \ref{RSW}, which we perform in Section 4. Finally, Section 5 is devoted to presenting the consequences that we mentioned.

\section{FK percolation background}

\subsection{Basic features of the model}

In order to remain as self-contained as possible, we recall some basic features of the random-cluster models. Some of these properties, like the Fortuin-Kasteleyn-Ginibre (FKG) inequality, are common to many statistical mechanics models. The reader can consult the reference book \cite{G_book_FK} for more details, and proofs of the results stated.

\subsubsection*{Definition of random-cluster measures}

The random-cluster (or \emph{FK percolation}) measure can be defined on any finite graph, but here we only consider finite subgraphs $G$ of the square lattice $(\mathbb{Z}^2,\mathbb{E}^2)$. We denote by $\partial G$ the boundary of such a subgraph $G$, that is, the vertices having \emph{less than four adjacent edges} -- notice that this definition is non standard. A \emph{configuration} $\omega$ is a random subgraph given by the vertices of $G$, together with some subset of edges between them. An edge of $G$ is called \emph{open} if it belongs to $\omega$, and \emph{closed} otherwise. Two sites $x$ and $y$ are said to be \emph{connected} if there is an \emph{open path} -- a path composed of open edges only -- connecting them, which is denoted by $x\leadsto y$. Similarly, two sets of vertices $X$ and $Y$ are said to be connected if there exist two sites $x \in X$ and $y \in Y$ such that $x \leadsto y$, and we use the notation $X \leadsto Y$. We also abbreviate $\{x\} \leadsto Y$ as $x \leadsto Y$. Sites can be grouped into (maximal) connected components, usually called \emph{clusters}.

Contrary to usual independent percolation, the edges in the FK percolation model are dependent of each other, a fact which makes the notion of \emph{boundary conditions} important. Formally, a set $\xi$ of boundary conditions is a set of ``abstract'' edges, each connecting two boundary vertices, that encodes how these vertices are connected outside $G$. We denote by $\omega \cup \xi$ the graph obtained by adding the new edges in $\xi$ to the configuration $\omega$.

We are now in a position to define the FK percolation measure itself, for any parameters $p \in [0,1]$ and $q\geq 1$. Denoting by $o(\omega)$ (resp. $c(\omega)$) the number of open (resp. closed) edges of $\omega$, and by $k(\omega,\xi)$ the number of connected components in $\omega \cup \xi$, the FK percolation process on $G$ with parameters $p$, $q$ and boundary conditions $\xi$ is obtained by taking 
\begin{equation} \label{def_FK}
\P_{p,q,G}^{\xi}(\{\omega\})=\frac{p^{o(\omega)}(1-p)^{c(\omega)}q^{k(\omega,\xi)}}{Z_{p,q,G}^{\xi}}
\end{equation}
as a probability for any configuration $\omega$ on $G$, where $Z_{p,q,G}^{\xi}$ is an appropriate normalizing constant, called the \emph{partition function}.

Among all the possible boundary conditions, two of them play a particular role. On the one hand, the \emph{free} boundary conditions correspond to the case when there are no extra edges connecting boundary vertices, we denote by $\P_{p,q,G}^0$ the corresponding measure. On the other hand, the \emph{wired} boundary conditions correspond to the case when all the boundary vertices are pair-wise connected, and the corresponding measure is denoted by $\P_{p,q,G}^1$.

\subsubsection*{Domain Markov property}

The different edges of an FK percolation model being highly dependent, what happens in a given domain depends on the configuration outside the domain. However, the FK percolation model possesses a very convenient property known as the \emph{Domain Markov property}, which usually makes it possible to obtain some spatial independence. This property is really instrumental in all our proofs.

Consider a graph $G$, with $E$ its set of vertices. For a subset $F \subseteq E$, consider the graph $G'$ having $F$ as a set of vertices, and the edges of $G$ connecting sites of $F$ as a set of edges. Then for any boundary conditions $\phi$, $\mathbb{P}^{\phi}_{p,q,G}$ conditioned to match some configuration $\omega$ outside $G$ is equal to $\mathbb{P}_{p,q,G'}^{\xi}$, where $\xi$ is the set of connections inherited from $\omega$. In other words, one can encode, using appropriate boundary conditions $\xi$, the influence of the configuration outside $G$.

\subsubsection*{Strong positive association and infinite-volume measures}

The random-cluster model with parameters $p \in [0,1]$ and $q \geq 1$ on a finite graph $G$ has the \emph{strong positive association property}. More precisely, it satisfies the so-called Holley criterion, a fact which has two important consequences. A first consequence is the well-known \emph{FKG inequality}
\begin{equation}
\P^{\xi}_{p,q,G}(A\cap B)\geq \P^{\xi}_{p,q,G}(A) \: \P^{\xi}_{p,q,G}(B)
\end{equation}
for any pair of \emph{increasing} events $A$, $B$ (increasing events are defined in the usual way \cite{G_book_FK}) and any boundary conditions $\xi$. This correlation inequality is fundamental to study FK percolation, for instance to combine several increasing events such as the existence of crossings in various rectangles. 

A second property implied by the strong positive association is the following monotonicity between boundary conditions, which is particularly useful when combined with the Domain Markov property. For any boundary conditions $\phi\leq \xi$ (all the connections present in $\phi$ belong to $\xi$ as well), we have
\begin{equation} \label{comparison between boundary conditions}
\mathbb{P}^{\phi}_{p,q,G}(A)\leq \mathbb{P}^{\xi}_{p,q,G}(A)
\end{equation}
for any increasing event $A$ that depends only on $G$. We say that $\mathbb{P}^{\phi}_{p,q,G}$ is stochastically dominated by $ \mathbb{P}^{\xi}_{p,q,G}$ (denoted by $\mathbb{P}^{\phi}_{p,q,G} \leqst \mathbb{P}^{\xi}_{p,q,G}$).

In particular, this property directly implies that the free and wired boundary conditions are extremal in the sense of stochastic ordering: for any set of boundary conditions $\xi$, one has
\begin{equation} \label{extremality}
\P_{p,q,G}^0 \leqst \P_{p,q,G}^{\xi} \leqst \P_{p,q,G}^1.
\end{equation}

An infinite-volume measure can be constructed as the increasing limit of FK percolation measures on the nested sequence of graphs $(\llbracket-n,n\rrbracket^2)_{n \geq 1}$ with free boundary conditions. For any fixed $q\geq 1$, classical arguments then show that there must exist a critical point $p_c=p_c(q)$ such that for any $p<p_c$, there is almost surely no infinite cluster of sites, while for $p > p_c$, there is almost surely one (see \cite{G_book_FK} for example).

\subsubsection*{Planar duality}

In two dimensions, an FK measure on a subgraph $G$ of $\mathbb{Z}^2$ with free boundary conditions can be associated with a dual measure in a natural way, as we explain now. The dual graph $G^*$ is obtained by putting a vertex at the center of each face of $\mathbb{Z}^2$ having an edge in $G$. The edges are connecting any two adjacent vertices for which the corresponding faces are separated by an edge of $G$. The FK percolation model $\P_{p,q,G}^{0}$ is then dual to the measure $\P_{p^*,q,G^*}^{1}$, where $p^*$ satisfies
\begin{equation}
\frac{p p^*}{(1-p) (1-p^*)} = q .
\end{equation}
One then expects the critical point $p_c(q)$ to be the self-dual point $p_{\textrm{sd}}(q)$ for which $p = p^*$, whose value can easily be derived:
\begin{equation}
p_{\textrm{sd}}(q)=\frac{\sqrt{q}}{1+\sqrt{q}}.
\end{equation}

\subsubsection*{FK percolation with parameter $q=2$: FK Ising model}

For the value $q=2$ of the parameter, the FK percolation model is related to the Ising model. More precisely, if starting from an FK percolation sample, one assigns uniformly at random a spin $+1$ or $-1$ to each cluster as a whole (sites in the same cluster get the same spin), independently, we get simply a sample of the Ising model. This coupling is called the \emph{Edwards-Sokal coupling} \cite{edwards-sokal}.

In this case, the FK percolation model is now well-understood. The value $p_c= p_{\mathrm{sd}}$ is implied by the computation by Kaufman and Onsager \cite{KaufmanOnsager} of the partition function of the Ising model, and an alternative proof has been proposed recently by Beffara, Duminil-Copin and Smirnov \cite{BeffaraDuminilCopinSmirnov}. Moreover, in \cite{Sm3}, Smirnov proved conformal invariance of this model at the self-dual point $ p_{\mathrm{sd}}$.

In the following, we restrict ourselves to the FK percolation model with parameters $q=2$ and $p= p_{\mathrm{sd}}(2)=\sqrt{2}/(1+\sqrt{2})$ (so that we forget the dependence on $p$ and $q$), which is also known as the \emph{critical FK Ising model} -- we often call it the FK Ising model for short. 

\subsection{Smirnov's fermionic observable}

In this part, we recall discrete analyticity and discrete harmonicity results for the FK Ising model, established by Smirnov in \cite{Sm3}. These results are crucial in our proofs since they will allow us to compare connection probabilities to harmonic measures. \emph{Recall that from now, $q=2$ and $p= p_{\mathrm{sd}}(2)$}.

\subsubsection*{FK Ising model in Dobrushin domains}

Let $\mathcal{D}$ be a \emph{finite} subgraph of the primal lattice $\mathbb{Z}^2$ such that $\partial_e\mathcal{D}$ is a self-avoiding polygon-- where $\partial_e \mathcal{D}$ is the set of edges between boundary sites. Hence, $Int(\mathcal{D})$, the connected component of $\mathbb{R}^2\setminus \partial_e\mathcal{D}$ containing the graph, is a bounded and simply connected domain. Consider two sites $a$ and $b$ of the boundary of $\mathcal{D}$. They determine two arcs of $\partial _e\mathcal{D}$, $(ab)$ and $(ba)$, obtained by following $\partial_e \mathcal{D}$ in the counterclockwise direction from $a$ to $b$, and conversely.

We consider a random cluster measure with free boundary conditions on $(ab)$, and wired boundary conditions on $(ba)$. These boundary conditions are called the \emph{Dobrushin boundary conditions} on $(\mathcal{D},a,b)$, $(ab)$ is called the \emph{free} arc and $(ba)$ the \emph{wired arc}. We denote by $\mathbb{P}_{\mathcal{D},a,b}$ the associated random cluster measure with parameters $q=2$ and $p= p_{\mathrm{sd}}(2)$. This measure has a very nice representation on the so-called medial graph of $\mathcal{D}$, which we define now.

\subsubsection*{Medial lattice and loop representation for the FK Ising model}

We first define the \emph{medial lattice} associated with the square lattice $\mathbb{Z}^2$. In order to do that, consider together $\mathbb{Z}^2$ with its dual $(\mathbb{Z}^2)^*$, and declare \emph{black} the sites of the primal lattice $\mathbb{Z}^2$, and \emph{white} the sites of the dual lattice $(\mathbb{Z}^2)^*$. We then introduce the graph obtained by replacing every site by a colored diamond, as on Figure \ref{fig:medial_lattice}. We obtain in this way a rotated copy of the square lattice (scaled by a factor $1/\sqrt{2}$), denoted by $(\mathbb{Z}^2)_{\diamond}$. The sites of the primal (resp. dual) lattice are thus associated with the black (resp. white) faces: we use extensively in the proof this correspondence between sites of the primal and of the dual lattices, and faces of the medial lattice. For instance, we say that two black diamonds are connected if the corresponding sites of the primal lattice are connected.

\begin{figure}
\begin{center}
\includegraphics[width=8.5cm]{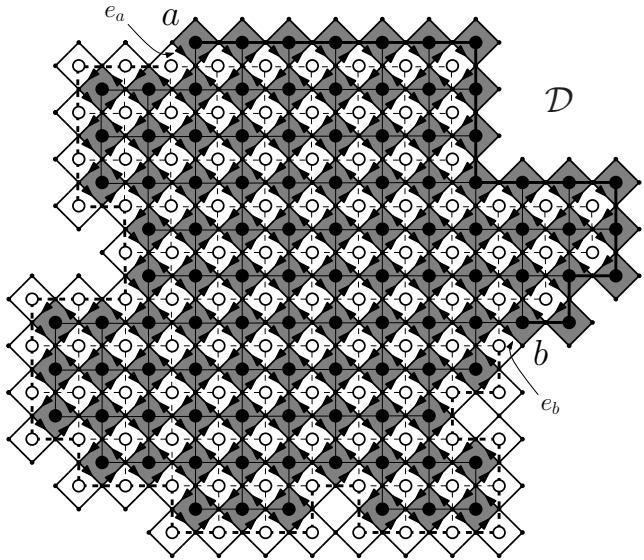}
\caption{\label{fig:medial_lattice}A domain $\mathcal{D}$ with Dobrushin boundary conditions: the vertices of the primal graph are black, the vertices of the dual graph $\mathcal{D}^*$ are white, and between them lies the medial lattice $\mathcal{D}_{\diamond}$.}
\end{center}
\end{figure}

We construct the \emph{medial graph} $\mathcal{D}_{\diamond}$ of $(\mathcal{D},a,b)$ by considering the diamonds intersecting $Int(\mathcal{D})$, together with the white diamonds touching the free arc $(ba)$ (see Figure \ref{fig:medial_lattice}). These white diamonds form the \emph{free arc} of $\mathcal{D}_{\diamond}$, the black diamonds corresponding to sites of the wired arc of $\mathcal{D}$ form the \emph{wired arc} of $\mathcal{D}_{\diamond}$. The corners of diamonds \emph{not belonging} to the two arcs of $\mathcal{D}_{\diamond}$ define the \emph{vertices} of $\mathcal{D}_{\diamond}$. The \emph{edges} are the edges of $(\mathbb{Z}^2)_{\diamond}$ between two vertices.
We adopt the following convenient convention: the two edges $e_a$ and $e_b$ of $(\mathbb{Z}^2)_{\diamond}$ (resp. adjacent to $a$ and $b$) that ``separate'' the free and the wired arcs of $\mathcal{D}_{\diamond}$ are considered as edges of $\mathcal{D}_{\diamond}$.

\begin{remark}The two definitions of arcs (for $\mathcal{D}$ and $\mathcal{D}_{\diamond}$) are quite similar. Nevertheless, the free arc of $\mathcal{D}$ is composed of sites of $\mathbb{Z}^2$ while the free arc of $\mathcal{D}_{\diamond}$ is composed of white diamonds of $\mathcal{D}_{\diamond}$. Moreover, vertices of $\mathcal{D}_{\diamond}$ possess two adjacent edges if they are ``on the boundary'' (except next to $a$ and $b$), and four otherwise.
\end{remark}

For any FK percolation configuration in $\mathcal{D}$, the interfaces between the primal clusters and the dual clusters (if we follow the edges of the medial lattice) form a family of loops, together with one path from $e_a$ to $e_b$, called the \emph{exploration path}, as shown on Figure \ref{fig:loop_configuration}. A simple rearrangement of (\ref{def_FK}) shows that the probability of such a configuration is proportional to $(\sqrt{2})^{\#\text{loops}}$ -- taking into account the fact that $q=2$ and $p=p_{\mathrm{sd}}(2)$.

In addition to this, we put an orientation on the medial graph: we orient the edges of each black face in such a way that the arrows are in counter-clockwise order. It naturally gives an orientation to the loops, so that we are now working with a model of oriented curves on the medial lattice.

\begin{remark}
If we consider a Dobrushin domain $(\mathcal{D},a,b)$, the slit domain created by ``removing'' the $T$ first steps of the exploration path is again a Dobrushin domain. More precisely, consider the new arc $l$ composed of $\partial_e\mathcal{D}$, together with the sites of $\mathcal{D}$ adjacent to the exploration path. We can define a new domain by removing all the sites of $\mathcal{D}$ which are not in the same connected component of $\mathcal{D}\setminus l$ as $b$: we obtain a new Dobrushin domain $(\mathcal{D}\setminus \gamma[0,T],\gamma(T),b)$, where, with a slight abuse of notation, $\gamma(T)$ is used to denote the site of the primal lattice adjacent to the medial edge $\gamma(T)$. The exploration path $\gamma$ is the interface between the primal open cluster connected to the wired arc and the dual open cluster connected to the free arc, so that, conditionally on $\gamma$, the law of the FK Ising model in the new domain is exactly $\mathbb{P}_{\mathcal{D}\setminus \gamma[0,T],\gamma(T),b}$. This observation will be instrumental in our proof.
\end{remark}

\begin{figure}
\begin{center}
\includegraphics[width=8cm]{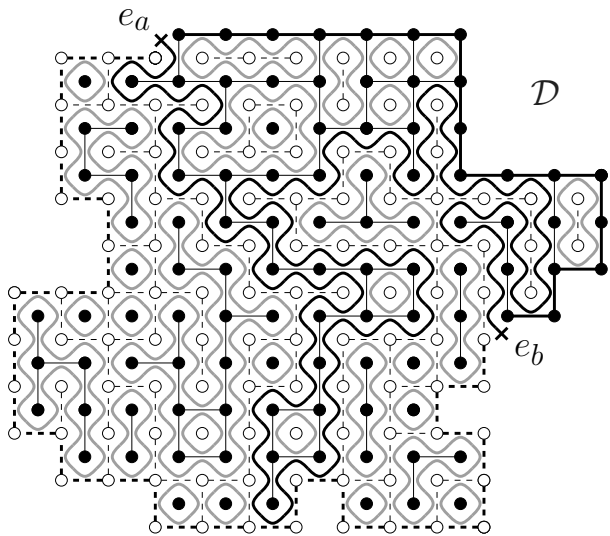}
\caption{\label{fig:loop_configuration}An FK percolation configuration in the Dobrushin domain $(\mathcal{D},a,b)$, together with the corresponding interfaces on the medial lattice: the loops in grey, and the exploration path $\gamma$ from $e_a$ to $e_b$ in black.}
\end{center}
\end{figure}

\subsubsection*{Fermionic observable and local relations}

Let $(\D,a,b)$ be a Dobrushin domain and $\g$ the exploration path from $e_a$ to $e_b$. The \emph{winding} $W_{\Gamma}(z,z')$ of a curve $\Gamma$ between two edges $z$ and $z'$ of the medial lattice is the total rotation (in radians) that the curve makes from the center of the edge $z$ to the center of the edge $z'$. The \textit{fermionic observable} $F$ can now be defined by the formula \cite{Sm3}
\begin{equation} \label{observable_F}
F(e)=\mathbb{E}_{\D,a,b}[{\rm e}^{-\frac{1}{2}\cdot{\rm i} W_{\gamma}(e_a,e)} \mathbb{I}_{e\in \gamma}],
\end{equation}
for any edge $e$ of the medial lattice $\mathcal{D}_{\diamond}$. The constant $\sigma=1/2$ appearing in front of the winding is called the \emph{spin} (see \cite{Sm3}). 

The quantity $F(e)$ is a complexified version of the probability that $e$ belongs to the exploration path (note that it is defined on the medial graph $\mathcal{D}_{\diamond}$). The complex weight makes the link between $F$ and probabilistic properties less explicit. Nevertheless, as we will see, the winding term can be controlled close to the boundary. The observable $F$ also satisfies the following local relation, from which Propositions \ref{definitionH} and \ref{subsupharm} follow. 

\begin{lemma}[\cite{Sm3}]\label{Cauchy-Riemann}
For any vertex $v$ of the medial lattice $\D_{\diamond}$ with \emph{four adjacent edges} in $\mathcal{D}_{\diamond}$, the relation
\begin{equation} \label{eq:cauchy_riemann}
F(e_1)+F(e_3) = F(e_2)+F(e_4)
\end{equation}
is satisfied, where $e_1$, $e_2$, $e_3$ and $e_4$ are the four edges at $v$ indexed in clockwise order, as on Figure \ref{fig:cauchy_riemann}.
\end{lemma}

\begin{figure}
\begin{center}
\includegraphics[width=3cm]{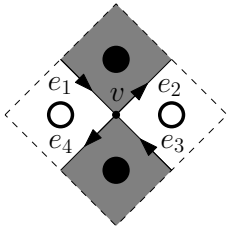}
\caption{\label{fig:cauchy_riemann}The discrete relation at a vertex $v$.}
\end{center}
\end{figure}

We refer to \cite{Sm3} or \cite{BeffaraDuminilCopinSmirnov} for the proof of this result. The key ingredient is a bijection between configurations that contribute to the values of $F$ at the vertices around $v$. Note that for other values of $q$, one can still define the fermionic observable in a way similar to Eq.(\ref{observable_F}): for an appropriate value $\sigma = \sigma(q)$ of the spin, the previous relation Eq.(\ref{eq:cauchy_riemann}) still holds (see \cite{Sm3}, \cite{BeffaraDuminilCopinSmirnov}, or \cite{DuminilExponentialDecay}).

\subsubsection*{Complex argument of the fermionic observable $F$ and definition of $H$}

Due to the specific value of the spin $\sigma=1/2$, corresponding to the value $q=2$, the complex argument modulo $\pi$ of the fermionic observable $F$ follows from its definition Eq.(\ref{observable_F}). For instance, if the edge $e$ points in the same direction as the starting edge $e_a$, then the winding is a multiple of $2\pi$, so that the term ${\rm e}^{-\frac{{\rm i}}{2} W_{\gamma}(e_a,z)}$ is equal to $\pm 1$, and $F(e)$ is purely real. The same reasoning can be applied to any edge to show that it belongs to the line ${\rm e}^{{\rm i}\pi/4}\mathbb{R}$, ${\rm e}^{-{\rm i}\pi/4}\mathbb{R}$ or ${\rm i}\mathbb{R}$ depending on its direction. Contrary to Lemma \ref{Cauchy-Riemann}, this property is very specific to the FK Ising model.

For a vertex $v$ with \emph{four} adjacent edges, keeping the same notations as in the previous subsection, $F(e_1)$ and $F(e_3)$ are always orthogonal (for the scalar product between complex numbers $(a,b)\mapsto \Re e(a\bar{b})$), as well as $F(e_2)$ and $F(e_4)$, so that Eq.(\ref{eq:cauchy_riemann}) can be rewritten as
\begin{equation} \label{relationH1}
\left|F(e_1)\right|^2+\left|F(e_3)\right|^2=\left|F(e_2)\right|^2+\left|F(e_4)\right|^2.
\end{equation}

Consider now a vertex $v$ with \emph{two} adjacent edges of $\mathcal{D}_{\diamond}$, and denote by $e_5$ the ``entering'' edge, and $e_6$ the ``exiting'' edge. Such a vertex must be on the boundary of the domain, and $e_5$ belongs to $\g$ if and only if $e_6$ belongs to $\g$ -- indeed, by construction, the curve entering through $e_5$ must leave through $e_6$. Moreover, the windings of the curve $W_{\g}(e_a,e_5)$ and $W_{\g}(e_a,e_6)$ are constant since $\g$ cannot wind around these edges. From these two facts, we deduce:
\begin{equation} \label{relationH2}
|F(e_5)|^2=\left|{\rm e}^{-\frac{{\rm i}}{2}W_{\g}(e_a,e_5)}\mathbb{P}_{\mathcal{D},a,b}(e_5\in \g)\right|^2=\mathbb{P}_{\mathcal{D},a,b}(e_5\in \g)^2=|F(e_6)|^2.
\end{equation}

From Eqs.(\ref{relationH1}) and (\ref{relationH2}), one can easily prove the following proposition.
\begin{proposition}[\cite{Sm3}]\label{definitionH}There exists a \emph{unique function} $H$ defined on the faces of $\D_{\diamond}$ by the relation
\begin{equation} \label{def_H}
H(B)-H(W)=\left|F(e)\right|^2,
\end{equation}
for any two neighboring faces $B$ and $W$, respectively black and white, separated by the edge $e$, and by fixing the value $1$ on the black face corresponding to $a$. Moreover, $H$ is then automatically equal to $1$ on the black diamonds of the wired arc, and equal to $0$ on the white diamonds of the free arc.
\end{proposition}
This function $H$ is a discrete analogue of the antiderivative of $F^2$, as explained in Remark 3.7 of \cite{Sm3}.

\subsubsection*{Approximate Dirichlet problem for $H$}

Let us denote by $H_{\bullet}$ and $H_{\circ}$ the restrictions of $H$ respectively to the black faces and to the white faces. For a black site of $\mathcal{D}$ not on the boundary, we can consider the usual \emph{discrete Laplacian} (on the graph $\mathcal{D}$) at this site: it is the average on the four nearest black neighbors, minus the value at the site itself. A similar definition holds for white sites of the graph $\mathcal{D}^*$.

The result below, proved in \cite{Sm3}, is a key step to prove convergence of the observable as one scales the domain -- but we will not discuss this question here. Its proof relies on an elementary yet quite lengthy computation. 

\begin{proposition}[\cite{Sm3}]\label{subsupharm}
The function $H_{\bullet}$ (resp. $H_{\circ}$) is \emph{subharmonic} (resp. \emph{superharmonic}) inside the domain for the discrete Laplacian.
\end{proposition}

Since we know that $H$ is equal to $1$ (resp. $0$) on the black diamonds of the wired arc (resp. the white diamonds of the free arc), the previous proposition can be seen as an approximate Dirichlet problem for the function $H$. In the next section, we make this statement rigorous by comparing $H$ to harmonic functions corresponding to the same boundary problems (on the set of black faces, or on the set of white ones).

\section{Comparison to harmonic measures} \label{comp-harm-mes}

In this section, we obtain a comparison result for the boundary values of the fermionic observable $F$ introduced
in the previous section in terms of discrete harmonic measures. It will be used to obtain all the quantitative
estimates on the observable that we need for the proof of Theorem \ref{RSW}.

\subsection{Comparison principle}

As in the previous section, let $(\D, a, b)$ be a discrete Dobrushin domain, with free boundary conditions on the
counterclockwise arc from $a$ to $b$ and wired boundary conditions on the counterclockwise arc from $b$ to $a$. 

For our estimates, we first extend the medial graph of our discrete domain by adding two extra layers of faces: 
one layer of white diamonds adjacent to the black diamonds of the sites of the wired arc, and one layer of 
black diamonds adjacent to the white diamonds of the free arc. We denote by $\bar{\D}_{\diamond}$ this extended domain.

\begin{remark}
Note that one faces a small technicality when adding a new layer of diamonds: some of these additional diamonds can overlap diamonds that were already here. For instance, if the domain has a slit, the free and the wired arc are adjacent along this slit, and the extra layer on the wired arc (resp. free arc) overlaps the free arc (resp. wired arc). As we will see, $H_{\bullet}$ is equal to $1$ on the wired arc, and to $0$ on the additional layer along the free arc. One should thus remember in the following that the added diamonds are considered as different from the original ones -- it will always be clear from the context which diamonds we are considering.
\end{remark}

\smallskip

For any given black face $B$, let us define $\left( X_{\bullet t}^B \right)_{t\geq 0}$ to be the continuous-time random walk 
on the black faces of $\bar{\D}_{\diamond}$ starting at $B$, that jumps with rate $1$ on adjacent black faces, \emph{except} for the black faces on the extra layer of black diamonds
adjacent to the free arc onto which it jumps with rate $\rho := (\sqrt{2} + 1) / 2$. Similarly, we denote by $\left( X_{\circ t}^{W}\right)_{t \geq 0}$ the continuous-time
random walk on the white faces of $\bar{\D}_{\diamond}$ starting at a white face $W$ that jumps with rate $1$ on adjacent white faces, \emph{except} for adjacent white faces on the extra layer of white diamonds adjacent to the wired arc onto which it jumps with the same rate $ \rho = (\sqrt{2} + 1) / 2$ as previously.

\smallskip

For a black face $B$, we denote by $\Hb(B)$ the probability that the random walk $ X_{\bullet t}^B $
hits the wired arc from $b$ to $a$ before hitting the extra layer adjacent to the free arc.
Similarly, for $W$ a white face, we denote by $\Hw(W)$ the probability
that the random walk $X_{\circ t}^W$ hits the additional layer adjacent to the wired arc before hitting the free arc. 
Note that there is no extra difficulty in defining these quantities for infinite discrete domains as well.
We have the following result:

\begin{proposition}[uniform comparability]\label{uniform comparability}
Let $(\D, a, b)$ be a discrete Dobrushin domain.
For any medial edge $e$ adjacent to a boundary edge of the
free arc, if we denote by $B=B(e)$ the black face that it borders and by $W=W(e)$ any closest white face that is \emph{not} 
on the free arc, we have
\begin{equation}
\sqrt{\Hw(W)} \leq |F (e)| \leq \sqrt{\Hb(B)} .
\end{equation}
\end{proposition}

\begin{proof}

By construction of the function $H$, we have $|F(e)|^2=H(B)$ and $H(W) = |F(e)|^2 - |F(e')|^2 \leq |F(e)|^2$, where $e'$ is the medial edge between $B$ and $W$: it is therefore sufficient to show that $H(B) \leq \Hb(B)$ and $ H(W) \geq \Hw(W)$. We only prove that $H(B)\leq \Hb(B)$, since the other case can be handled in the same way.

For this, we use a variation of a trick introduced in 
\cite{CS2} and extend the function $H$ to the extra layer of black diamonds -- added as explained
above -- by setting $H$ to be equal to $0$ there.
It is then sufficient to show that the restriction $H_{\bullet}$ of $H$ to the black diamonds of 
$\bar{\D}_{\diamond}$ is subharmonic for the Laplacian that is the generator
of the random walk $X_{\bullet}$, since it has the
same boundary values as $\Hb$. Inside the domain, this is given by
Proposition \ref{subsupharm}, since there the Laplacian is the usual discrete Laplacian (associated with it is just a simple random walk). The only thing to check is when a face involved in the computation of the Laplacian
belongs to one of the extra layers. For the sake of simplicity, we study the case when only one face belongs to these extra layers.

\begin{figure}
\begin{center}
\includegraphics[width=5cm]{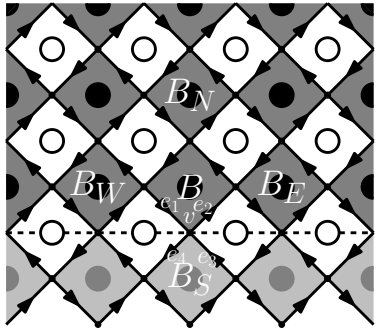}
\caption{\label{fig:comparison}We extend $\D_{\diamond}$ by adding two extra layers of medial faces, and extend the functions $H_{\bullet}$ and $H_{\circ}$ there. Here is represented the extension along the free arc.}
\end{center}
\end{figure}

Denote by $B_W$, $B_N$, $B_E$ and $B_S$ the black faces adjacent to $B$, and assume that $B_S$ is on the extra layer (see Figure \ref{fig:comparison}). The discrete Laplacian of $X_{\bullet}$ at face $B$ is denoted by $\Delta_{\bullet}$. We claim that
\small\begin{equation}\label{modified_laplacian}
\Delta_{\bullet} H_{\bullet}(B)=\frac{2+\sqrt{2}}{6+5\sqrt{2}}[H_{\bullet}(B_W)+H_{\bullet}(B_N)+H_{\bullet}(B_E)]+\frac{2\sqrt{2}}{6+5\sqrt{2}}H_{\bullet}(B_S)-H_{\bullet}(B)\geq 0.
\end{equation}\normalsize
For that, let us denote by $e_1, e_2, e_3, e_4$ the four medial edges at the bottom vertex $v$ between $B$ and $B_S$, in clockwise order, with $e_1$ and $e_2$ along $B$, and $e_3$ and $e_4$ along $B_S$ (see Figure \ref{fig:comparison}) -- note that $e_3$ and $e_4$ are not edges of $\mathcal{D}_{\diamond}$, but of $(\mathbb{Z}^2)_{\diamond}$.

We extend $F$ to $e_3$ and $e_4$ by requiring $F(e_3)$ and $F(e_1)$ to be orthogonal, as well as $F(e_4)$ and $F(e_2)$, and $F(e_1) + F(e_3) =  F(e_2) + F(e_4)$ to hold true. This defines these two values uniquely: indeed, as noted before, we know that $F(e_2) = {\rm e}^{-{\rm i} \pi/4} F(e_1)$  on  the boundary (since $W_{\gamma}(e_a,e_1)$ and $W_{\gamma}(e_a,e_2)$ are fixed, with $W_{\gamma}(e_a,e_2) = W_{\gamma}(e_a,e_1) + \pi/2$, and the curve cannot go through one of these edges without going through the other one), which implies, after a small calculation, that
$$|F(e_3)|^2=\Big|\Big(\tan \frac{\pi}{8}\Big)  {\rm e}^{{\rm i} \pi/4} F(e_2) \Big|^2 =\frac{2-\sqrt{2}}{2+\sqrt{2}} |F(e_2)|^2 = \frac{2-\sqrt{2}}{2+\sqrt{2}} H_{\bullet}(B).$$
If we denote by $\tilde{H}_{\bullet}$ the function defined by $\tilde{H}_{\bullet}=H_{\bullet}$ on $B$, $B_W$, $B_N$ and $B_E$, and by 
\begin{equation}
\tilde{H}_{\bullet}(B_S)=|F(e_3)|^2=\frac{2-\sqrt{2}}{2+\sqrt{2}}H_{\bullet}(B),
\end{equation}
then $\tilde{H}_{\bullet}$ satisfies the same relation Eq.(\ref{def_H}) (definition of $H$) for $e_3$ and $e_4$, as inside the domain. Since the fermionic observable $F$ verifies the same local equations, the computation performed in the Appendix C of \cite{Sm3} is valid, Proposition \ref{subsupharm} applies at $B$, and we deduce
\begin{equation}\Delta \tilde{H}_{\bullet}(B)=\frac{1}{4}[\tilde{H}_{\bullet}(B_W)+\tilde{H}_{\bullet}(B_N)+\tilde{H}_{\bullet}(B_E)+\tilde{H}_{\bullet}(B_S)]-\tilde{H}_{\bullet}(B)\geq 0.
\end{equation}
Using the definition of $\tilde{H}_{\bullet}$, this inequality can be rewritten as
\begin{equation}\label{normal_Laplacian}
\frac{1}{4}[H_{\bullet}(B_W)+H_{\bullet}(B_N)+H_{\bullet}(B_E)]-\frac{6+5\sqrt{2}}{4(2+\sqrt{2})}H_{\bullet}(B)\geq 0.
\end{equation}
Now using that $H_{\bullet}(B_S)=0$, we get the claim, Eq.(\ref{modified_laplacian}).
\end{proof}

\subsection{Estimates on harmonic measures}

In the previous subsection, we gave a comparison principle between the values of $H$ near the boundary, and the harmonic measures associated with two (almost simple) random walks, on the two lattices composed of the black faces and of the white faces respectively. In this subsection, we give estimates for these two harmonic measures in different domains needed for the proof of Theorem \ref{RSW}. We start by giving a lower bound which is useful in the proof of the 1-point estimate.

\begin{lemma} \label{2 points estimate}
For $\beta>0$ and $n\geq 0$, let $R_n^{\beta}$ be the graph
$$R_n^{\beta}=\llbracket -\beta n,\beta n\rrbracket\times \llbracket 0,2n\rrbracket.$$
Then there exists $c_1(\beta)>0$ such that for any $n\geq 1$, 
\begin{equation}\label{lower bound rectangle}
\Hw(W_x)\geq \frac{c_1(\beta)}{n^2}
\end{equation}
in the Dobrushin domain $(R_{n}^{\beta},u,u)$ (see Figure \ref{fig:example}), for all $x=(x_1,0)$ and $u=(u_1,2n)$ such that $|x_1|, |u_1| \leq \beta n/2$ (\emph{i.e.} far enough from the corners), $W_x$ being any of the two white faces that are both adjacent to $x$ and not on the free arc.
\end{lemma}

\begin{proof}
This proposition follows from standard results on simple random walks (the local central limit theorem and gambler's ruin type estimates). For the sake of conciseness, we do not provide a detailed proof.
\end{proof}

\begin{figure}
\begin{center}
\includegraphics[width=6cm]{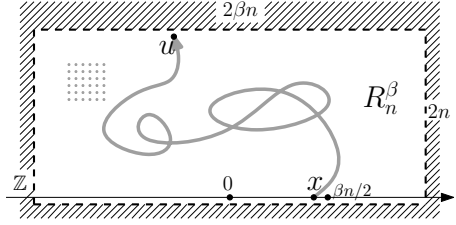}
\caption{\label{fig:example}Estimate of Lemma \ref{2 points estimate}: the dashed line corresponds to the free arc.}
\end{center}
\end{figure}

In the remaining of the section, we consider \emph{only} Dobrushin domains $(\mathcal{D},a,b)$ that \emph{contain the origin} on the free arc, and are \emph{subsets} of the medial lattice $\mathbb{H}_{\diamond}$, where $\mathbb{H}=\{(x_1,x_2)\in \Z^2,x_2\geq 0\}$ denotes the upper half plane -- in this case, we say that \emph{$\mathcal{D}$ is a Dobrushin $\mathbb{H}$-domain}. For the following estimates on harmonic measures, the Dobrushin domains that we consider can also be infinite. We are interested in the harmonic measure of the wired arc seen from a given point: without loss of generality, we can assume that this point is just the origin. Let $B_0$ be the corresponding black diamond of the medial lattice, and $W_0$ be an adjacent white diamond which is not on the free arc.

We first prove a lower bound on the harmonic measure. For that, we introduce, for $k \in \Z$ and $n \geq 0$, the segments
$$l_n(k) = \{ k \} \times \llbracket 0,n \rrbracket \quad (= \{(k,j):0\leq j\leq n\}).$$

\begin{lemma} \label{lower_bound}
There exists a constant $c_2>0$ such that for any Dobrushin $\mathbb{H}$-domain $(\mathcal{D},a,b)$, we have
\begin{equation}
\Hw(W_0) \geq \frac{c_2}{k},
\end{equation}
provided that, in $\mathcal{D}$, the segment $l_k(-k)$ disconnects from the origin the intersection of the free arc with the upper half-plane (see Figure \ref{fig:comparison_lower_bound}).
\end{lemma}

\begin{proof}
We know that $l_k(-k)$ disconnects the origin from the part of the free arc that lies in the upper half-plane, let us thus consider the connected component of $\mathcal{D}\setminus l_k(-k)$ that contains the origin. In this new domain $\mathcal{D}_0$, if we put free boundary conditions along $l_k(-k)$, the harmonic measure of the wired arc is smaller than the harmonic measure of the wired arc in the original domain $\mathcal{D}$. On the other hand, the harmonic measure of the wired arc in $\mathcal{D}_0$ is larger than the harmonic measure of the wired arc in the slit domain $(\mathbb{H}\setminus l_k(-k),(-k,k),\infty)$, which has respectively wired and free boundary conditions to the left and to the right of $(-k,k)$ (see Figure \ref{fig:comparison_lower_bound}). Estimating this harmonic measure is straightforward, using the same arguments as before.
\end{proof}

\begin{figure}
\begin{center}
\includegraphics[width=12cm]{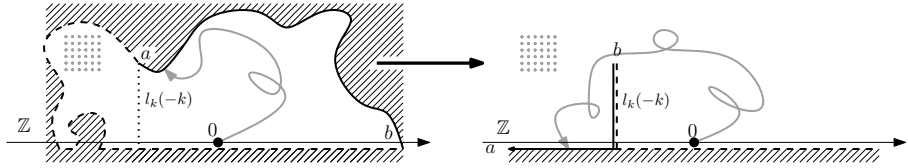}
\caption{\label{fig:comparison_lower_bound}The two domains involved in the proof of Lemma \ref{lower_bound}.}
\end{center}
\end{figure}

We now derive upper bounds on the harmonic functions. We will need two estimates of different types. The first one takes into account the distance between the origin and the wired arc, while the second one requires the existence of a segment $l_n(k)$ disconnecting the wired arc from the origin (still inside the domain).
 
\begin{lemma} \label{upper_bound}
There exist constants $c_3,c_4>0$ such that for any Dobrushin $\mathbb{H}$-domain $(\mathcal{D},a,b)$,

\begin{enumerate}[(i)]
\item \label{upperbounddistance} if $d_1(0)$ denotes the distance between the origin and the wired arc,
\begin{equation}
\Hb(B_0)\leq c_3 \frac{1}{d_1(0)},
\end{equation}

\item \label{upperboundarc} and if the segment $l_n(k)$ disconnects the wired arc from the origin inside $\mathcal{D}$,
\begin{equation}
\Hb(B_0) \leq c_4 \frac{n}{|k|^2}.
\end{equation}
\end{enumerate}
\end{lemma}

\begin{proof}
Let us first consider item (\ref{upperbounddistance}). For $d=d_1(0)$, define the Dobrushin domain $(\tilde{\mathcal{B}}_d,(-d,0),(d,0))$ where $\tilde{\mathcal{B}}_d$ is the set of sites in $\mathbb{H}$ at a graph distance at most $d$ from the origin (see Figure \ref{fig:upper_bound}). The harmonic measure of the wired arc in $(\mathcal{D},a,b)$ is smaller than the harmonic measure of the wired arc in this new domain $\tilde{\mathcal{B}}_d$, and, as before, this harmonic measure is easy to estimate.

Let us now turn to item (\ref{upperboundarc}). Since $l_n(k)$ disconnects the wired arc from the origin, the harmonic measure of the wired arc is smaller than the harmonic measure of $l_n(k)$ inside $\mathcal{D}$, and this harmonic measure is smaller than it is in the domain $\H\setminus l_n(k)$ with wired boundary conditions on the left side of $l_n(k)$ -- right side if $k<0$ (see Figure \ref{fig:upper_bound}). Once again, the estimates are easy to perform in this domain.
\end{proof}

\begin{figure}
\begin{center}
\includegraphics[width=12cm]{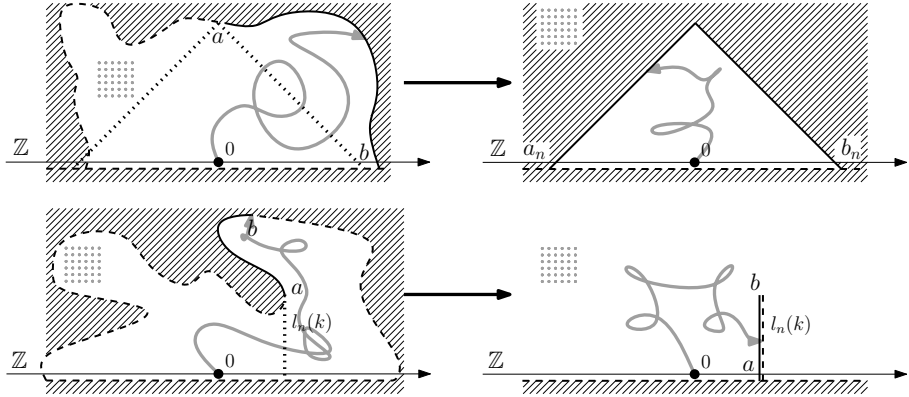}
\caption{\label{fig:upper_bound}The two different upper bounds (\ref{upperbounddistance}) and (\ref{upperboundarc}) of Lemma \ref{upper_bound}.}
\end{center}
\end{figure}

\section{Proof of Theorem \ref{RSW}} \label{section_proof}

We now prove our main result, Theorem \ref{RSW}. The main step is to prove the uniform lower bound for rectangles of bounded aspect ratio with free boundary conditions. We then use monotonicity to compare boundary conditions and obtain the desired result. In the case of free boundary conditions, the proof relies on a second moment estimate on the number $N$ of pairs of vertices $(x,u)$, on the top and bottom sides of the rectangle respectively, that are connected by an open path. 

The organization of this section follows the second-moment estimate strategy. In Proposition \ref{probability that one point belongs to gamma}, we first prove a lower bound on the probability for one site on the bottom side of a rectangle to be connected to a site on the top side. This estimate gives a lower bound on the expectation of $N$. Then, Proposition \ref{probability that two points belong to gamma} provides an upper bound on the probability that two points on the bottom side of a rectangle are connected to the top side. This proposition is the core of the proof, and it provides the right bound for the second moment of $N$. It allows us to conclude the section by using the second moment estimate method, thus proving Theorem \ref{RSW}.

In this section, we use two main tools: the Domain Markov property, and probability estimates for connections between the wired arc and sites on the free arc. We first explain how the previous estimates on harmonic measures can be used to derive estimates on connection probabilities: the following lemma is instrumental in our proof.

\begin{lemma}\label{interpretation H}
Let $(\mathcal{D},a,b)$ be a Dobrushin domain. For any site $x$ on the free arc of $\mathcal{D}$, we have
\begin{equation}
\sqrt{\Hw(W_x)}\leq \mathbb{P}_{\mathcal{D},a,b}(x\leadsto \text{wired arc}) \leq \sqrt{\Hb(B_x)},
\end{equation}
where $B_x$ is the black face corresponding to $x$, and $W_x$ is any closest white face that is not on the free arc. 
\end{lemma}

\begin{proof}
Since $x$ is on the free boundary of $\mathcal{D}$, there exists a white diamond on the free arc of $\mathcal{D}_{\diamond}$ which is adjacent to $B_x$: we denote by $e$ the edge between these diamonds. As noted before, since the edge $e$ is along the free arc, the winding $W_{\gamma}(e_a,e)$ of the exploration path $\gamma$ at $e$ is constant, and depends only on the direction of $e$. This implies that
$$\mathbb{P}_{\mathcal{D},a,b}(e\in \gamma)=|F(e)|.$$
In addition, $e$ belongs to $\gamma$ if and only if $x$ is connected to the wired arc, which implies that $|F(e)|$ is exactly equal to $\mathbb{P}_{\mathcal{D},a,b}(x\leadsto \text{wired arc})$. Proposition \ref{uniform comparability} thus implies the claim.
\end{proof}

With this lemma at our disposal, we can prove the different estimates. Throughout the proof, we use the notation $c_i(\beta)$ for constants that depend neither on $n$ nor on sites $x$, $y$ or on boundary conditions. When they do not depend on $\beta$, we denote them by $c_i$ (it is the case for the upper bounds). Recall the definition of $R_n^{\beta}$:
\begin{equation}
R_n^{\beta} = \llbracket -\beta n,\beta n \rrbracket \times \llbracket 0,2n\rrbracket.
\end{equation}
Let $\partial_+R_n^{\beta}$ (resp. $\partial_-R_n^{\beta}$) be the \emph{top side} $\llbracket -\beta n,\beta n \rrbracket \times\{2n\}$ (resp. \emph{bottom side} $\llbracket -\beta n,\beta n \rrbracket \times\{0\}$) of the rectangle $R_n^{\beta}$. We begin with a lower bound on connection probabilities.

\begin{proposition}[connection probability for one point on the bottom side]\label{probability that one point belongs to gamma}
Let $\beta>0$, there exists a constant $c(\beta)>0$ such that for any $n\geq 1$,
\begin{equation}
\P_{R_n^{\beta}}^0(x\leadsto u)\geq \frac{c(\beta)}{n}
\end{equation}
for all $x=(x_1,0)\in \partial_-R_n^{\beta}$, $u=(u_1,2n)\in \partial_+R_n^{\beta}$, satisfying $|x_1|, |u_1| \leq\beta n/2$.
\end{proposition}

\begin{proof}
The probability that $x$ and $u$ are connected in the rectangle with free boundary conditions can be written as the probability that $x$ is connected to the wired arc in $(R_n^{\beta},u,u)$ (where the wired arc consists of a single vertex). The previous lemma, together with the estimate of Lemma \ref{2 points estimate}, concludes the proof.
\end{proof}

We now study the probability that two boundary points on the bottom edge of $R_n^{\beta}$ are connected to the top edge.

\begin{proposition}[connection probability for two points on the bottom side]\label{probability that two points belong to gamma}
There exists a constant $c>0$ (uniform in $\beta,n$) such that for any rectangle $R_n^{\beta}$ and any two points $x,y$ on the bottom side $\partial_-R_n^{\beta}$,
\begin{equation}
\P_{R_n^{\beta},a_n,b_n}(x,y\leadsto \text{wired arc})\leq \frac{c}{\sqrt{|x-y|n}},
\end{equation}
where $a_n$ and $b_n$ denote respectively the top-left and top-right corners of the rectangle $R_n^{\beta}$. 
\end{proposition}

The proof is based on the following lemma, which is a strong form of the so-called half-plane one-arm probability estimate (see Section \ref{one_arm_exponent} for a further discussion of this result). For $x$ on the bottom side of $R_n^{\beta}$ and $k\geq 1$, we denote by $\mathcal{B}_k(x)$ the box centered at $x$ with diameter $k$ for the \emph{graph distance}. We can now state the lemma needed:

\begin{lemma}\label{upper bound for Tk}
There exists a constant $c_5>0$ (uniform in $n$, $\beta$ and the choice of $x$) such that for all $k \geq0$,
\begin{equation} \label{strong_one_arm}
\P_{R_n^{\beta},a_n,b_n}(\mathcal{B}_k(x) \leadsto\text{wired arc})\leq c_5\sqrt{\frac{k}{n}}.
\end{equation}
\end{lemma}

\begin{figure}
\begin{center}
\includegraphics[width=6cm]{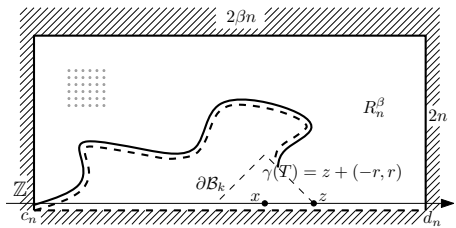}
\caption{\label{fig:image}The Dobrushin domain $(R_n^{\beta},c_n,d_n)$, together with the exploration path up to time $T$.}
\end{center}
\end{figure}

\begin{proof}
Consider $n,k,\beta>0$, and the box $R_n^{\beta}$ with one point $x\in \partial_-R_n^{\beta}$. Eq.(\ref{strong_one_arm}) becomes trivial if $k\geq n$, so we can assume that $k\leq n$. For any choice of $\beta'\geq \beta$, the monotonicity between boundary conditions Eq.(\ref{extremality}) implies that the probability that $\mathcal{B}_k(x)$ is connected to the wired arc $\partial_+R_n^{\beta}$ in $(R_n^{\beta},a_n,b_n)$ is smaller than the probability that $\mathcal{B}_k(x)$ is connected to the wired arc in the Dobrushin domain $(R_n^{\beta'},c_n,d_n)$, where $c_n$ and $d_n$ are the bottom-left and bottom-right corners of $R_n^{\beta'}$. From now on, we replace $\beta$ by $\beta+1$, and we work in the new domain $(R_n^{\beta},c_n,d_n)$. Notice that $\mathcal{B}_k$ is then included in $R_n^{\beta}$ and that the right-most site of $\mathcal{B}_k$ is at a distance at least $n$ from the wired arc.

We denote by $T$ the hitting time -- for the exploration path naturally para-metrized by the number of steps -- of the subset of the medial lattice composed of the edges adjacent to $\mathcal{B}_k(x)$; we set $T=\infty$ if the exploration path never reaches this set, so that $x$ is connected to $\mathcal{B}_k$ if and only if $T<\infty$.

Let $z$ be the right-most site of the box $\mathcal{B}_k(x)$. Consider now the event $\{z\leadsto \text{wired arc}\}$. By conditioning on the curve up to time $T$ (and on the event $\{\mathcal{B}_k(x)\leadsto \text{wired arc}\}$), we obtain
\small\begin{align*}
\P_{R_n^{\beta},c_n,d_n}(z \leadsto \text{wired arc})& =\E_{R_n^{\beta},c_n,d_n}\big[\mathbb{I}_{T<\infty}\P_{R_n^{\beta},c_n,d_n}(z\leadsto \text{wired arc} \:|\: \g[0,T])\big]\\
& = \E_{R_n^{\beta},c_n,d_n}\big[\mathbb{I}_{T<\infty}\P_{R_n^{\beta}\setminus \g[0,T],\g(T),d_n}(z\leadsto \text{wired arc})\big],
\end{align*}\normalsize
where in the second inequality, we have used the Domain Markov property, and also the fact that it is sufficient for $z$ to be connected to the wired arc in the new domain (since it is then automatically connected to the wired arc of the original domain). 

On the one hand, since $z$ is at a distance at least $n$ from the wired arc (thanks to the new choice of $\beta$), we can combine Proposition \ref{uniform comparability} and Lemma \ref{interpretation H} with item (\ref{upperbounddistance}) of Lemma \ref{upper_bound} to obtain
\begin{equation}\label{zz}
\mathbb{P}_{R_n^{\beta},c_n,d_n}(z\leadsto \text{wired arc})\leq \frac{c_5}{\sqrt{n}}.
\end{equation}
On the other hand, if $\gamma(T)$ can be written as $\gamma(T)=z+(-r,r)$, with $0 \leq r \leq k$, then the arc $z+l_r(-r)$ disconnects the free arc from $z$ in the domain $R_n^{\beta} \setminus \gamma[0,T]$, while if $\gamma(T)=z+(-r,2k-r)$, with $k+1 \leq r \leq 2k$, then the arc $z+l_r(-r)$ still disconnects the free arc from $z$. Using once again Proposition \ref{uniform comparability} and Lemma \ref{interpretation H}, this time with Lemma \ref{lower_bound}, we obtain that a.s.
\begin{equation}\label{zzz}
\P_{R_n^{\beta}\setminus \g(-\infty,T),\g(T),d_n}(z \leadsto \text{wired arc}) \geq \frac{c_4}{\sqrt{r}}\geq \frac{c_4}{\sqrt{2k}}.
\end{equation}
This estimate being uniform in the realization of $\gamma[0,T]$, we obtain
\begin{equation}
\frac{c_4}{\sqrt{2k}}\P_{R_n^{\beta},c_n,d_n}(T<\infty)\leq \P_{R_n^{\beta},c_n,d_n}(z\leadsto \text{wired arc})\leq \frac{c_5}{\sqrt{n}},
\end{equation}
which implies the desired claim Eq.(\ref{strong_one_arm}).
\end{proof}

\begin{figure}
\begin{center}
\includegraphics[width=12cm]{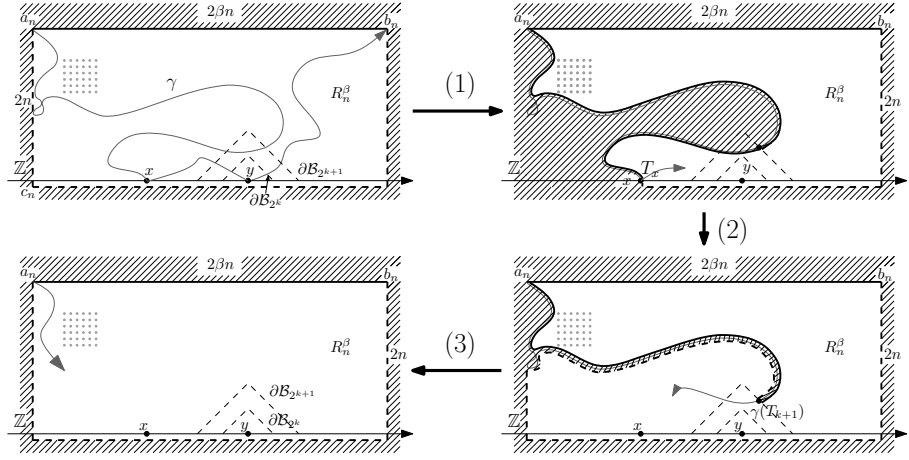}
\caption{\label{fig:image2}This picture presents the different steps in the proof of Proposition \ref{probability that two points belong to gamma}: we first (1) condition on $\gamma[0,T_x]$ and use the uniform estimate (\ref{upperbounddistance}) of Lemma \ref{upper_bound}, then (2) condition on $\gamma[0,T_{k+1}]$ and use the estimate (\ref{upperboundarc}) of Lemma \ref{upper_bound}, in order to (3) conclude with Lemma \ref{upper bound for Tk}.}
\end{center}
\end{figure}

\begin{proof}[Proof of Proposition \ref{probability that two points belong to gamma}]

Let us take two sites $x$ and $y$ on $\partial_-R_n^{\beta}$. As in the previous proof, the larger the $\beta$, the larger the corresponding probability, we can thus assume that $\beta$ has been chosen in such a way that there are no boundary effects. In order to prove the estimate, we express the event considered in terms of the exploration path $\g$. If $x$ and $y$ are connected to the wired arc, $\g$ must go through two boundary edges which are adjacent to $x$ and $y$, that we denote by $e_x$ and $e_y$. Notice that $e_x$ has to be discovered by $\g$ before $e_y$ is. 

We now define $T_x$ to be the hitting time of $e_x$, and $T_k$ to be the hitting time of the set of edges adjacent to the ball $\mathcal{B}_{2^k}(y)$, for $k \leq k_0=\lfloor \log_2|x-y| \rfloor$. If the exploration path does not cross this ball before hitting $e_x$, we set $T_k = \infty$. With these definitions, the probability that $e_x$ and $e_y$ are both on $\g$ can be expressed as
\begin{align}
&\P_{R_n^{\beta},a_n,b_n}(x,y\leadsto \text{wired arc})=\P_{R_n^{\beta},a_n,b_n}(e_x,e_y\in\g)\\
&=\sum_{k=0}^{k_0}\P_{R_n^{\beta},a_n,b_n}(e_y\in \g,T_x<\infty,T_{k+1}<T_{k}=\infty)\\
&=\sum_{k=0}^{k_0}\E_{R_n^{\beta},a_n,b_n}\big[\mathbb{I}_{T_{k+1}<T_k=\infty}\mathbb{I}_{T_x<\infty}\P_{R_n^{\beta},a_n,b_n}(e_y\in \g \left|\g[0,T_x]\right.)\big],
\label{estimation part 1}
\end{align}
where the second equality is obtained by conditioning on the exploration path up to time $T_x$. Recall that $e_y$ belongs to $\gamma$ if and only if $y$ is connected to the wired arc. Moreover, $y$ is at a distance at least $2^k$ from the wired arc in $R_n^{\beta}\setminus\gamma[0,T_x]$ (since $T_k=\infty$). Hence, the Domain Markov property and item (\ref{upperbounddistance}) of Lemma \ref{upper_bound} give that a.s.
\begin{equation}
\P_{R_n^{\beta},a_n,b_n}(e_y\in \g\left|\g[0,T_x]\right.)=\P_{R_n^{\beta}\setminus \g[0,T_x],x,b_n}(y\leadsto\text{wired arc})\leq \frac{c_3}{\sqrt{2^k}}.
\end{equation}
By plugging this uniform estimate into (\ref{estimation part 1}), and removing the condition on $T_k=\infty$, we obtain 
\small\begin{align*}
\P_{R_n^{\beta},a_n,b_n}(e_x,e_y\in\g)&\leq \sum_{k=0}^{k_0}\frac{c_3}{\sqrt{2^k}} \E_{R_n^{\beta},a_n,b_n}\big[\mathbb{I}_{T_{k+1}<\infty} \P_{R_n^{\beta},a_n,b_n}(T_x<\infty\left | \g[0,T_{k+1}]\right.)\big],
\end{align*}\normalsize
where we conditioned on the path up to time $T_{k+1}$. Now, $e_x$ belongs to $\gamma$ if and only if $x$ is connected to the wired arc, and in addition, the vertical segment connecting $\gamma(T_{k+1})$ to $\Z$, of length at most $2^{k+1}$, disconnects the wired arc from $x$ in the domain $R_n^{\beta}\setminus \gamma[0,T_{k+1}]$. Applying the Domain Markov property and item (\ref{upperboundarc}) of Lemma \ref{upper_bound}, we deduce that a.s.
\small\begin{align*}
\P_{R_n^{\beta},a_n,b_n}(e_x\in \g\left|\g[0,T_{k+1}]\right.)&=\P_{R_n^{\beta}\setminus \g[0,T_{k+1}],\g(T_{k+1}),b_n}(x\leadsto\text{wired arc})\leq c_4\frac{\sqrt{2^{k+1}}}{|x-y|}.
\end{align*}\normalsize
Making use of this uniform bound, we obtain
\begin{align*}
\P_{R_n^{\beta},a_n,b_n}(x,y\leadsto \text{wired arc})& \leq c_3 c_4 \sum_{k=0}^{k_0} \frac{\sqrt{2^{k+1}}}{\sqrt{2^k}|x-y|} \P_{R_n^{\beta},a_n,b_n}(T_{k+1}<\infty)\\
& \leq \frac{2c_3c_4c_5}{|x-y|\sqrt{n}}\sum_{k=0}^{k_0}\sqrt{2^{k}}\\
& \leq \frac{c}{\sqrt{n|x-y|}},
\end{align*}
using also Lemma \ref{upper bound for Tk} for the second inequality.

\end{proof}

We are now in a position to prove our main result.

\bigskip

\begin{proof}[Proof of Theorem \ref{RSW}]

Let $\beta>0$ and $n>0$, and also $R_n^{\beta}$ defined as previously. 

\bigskip

\noindent \textbf{Step 1: lower bound for free boundary conditions.} Let $N_n$ be the number of connected pairs $(x,u)$, with $x\in \partial_-R_n^{\beta}$, and $u\in \partial_+R_n^{\beta}$. The expected value of this quantity is equal to
\begin{equation}
\E_{R_n^{\beta}}^0[N_n]=\sum_{\substack{u\in \partial_+R_n^{\beta}\\ x\in \partial_-R_n^{\beta}}}\P_{R_n^{\beta}}^0(x \leadsto u).
\end{equation}
Proposition \ref{probability that one point belongs to gamma} directly provides the following lower bound on the expectation by summing on the $(\beta n)^2$ pairs of points $(x,u)$ far enough from the corners, \emph{i.e.} satisfying the condition of the proposition:
\begin{equation}
\E_{R_n^{\beta}}^0[N_n]\geq c_6(\beta)n
\end{equation}
for some $c_6(\beta) >0$.

On the other hand, if $x$ and $u$ (resp. $y$ and $v$) are pair-wise connected, then they are also connected to the horizontal line $\Z\times\{n\}$ which is (vertically) at the middle of $R_n^{\beta}$. Moreover, the Domain Markov property implies that the probability -- in $R_n^{\beta}$ with free boundary conditions -- that $x$ and $y$ are connected to this line is smaller than the probability of this event in the rectangle of half height with wired boundary conditions on the top side. In the following, we assume without loss of generality that $n$ is even and we set $m=n/2$, so that the previous rectangle is $R^{2\beta}_m$, and we define $a_m$ and $b_m$ as before. Using the FKG inequality, and also the symmetry of the lattice, we get
\small\begin{align*}
\P_{R_n^{\beta}}^0(x\leadsto u & ,y\leadsto v)& \leq \P_{R_{m}^{2\beta},a_m,b_m}(x,y\leadsto \text{wired arc})\:\P_{R_m^{2\beta},a_m,b_m}(\bar{u},\bar{v}\leadsto \text{wired arc}),
\end{align*}\normalsize
where $\bar{u}$ and $\bar{v}$ are the projections on the real axis of $u$ and $v$. Summing the bound provided by Proposition \ref{probability that two points belong to gamma} on all sites $x,y\in \partial_-R_n^{\beta}$ and $u,v\in\partial_+R_n^{\beta}$, we obtain
\begin{equation}
\E_{R_n^{\beta}}^0[N_n^2]\leq c_7 m^2\leq c_7n^2
\end{equation}
for some constant $c_7>0$. Now, by the Cauchy-Schwarz inequality,
\begin{equation}
\P_{R_n^{\beta}}^0(\C_v(R_n^{\beta}))=\P_{R_n^{\beta}}^0(N_n>0) = \E_{R_n^{\beta}}^0[(\mathbb{I}_{N_n>0})^2] \geq \frac{\E_{R_n^{\beta}}^0[N_n]^2}{\E_{R_n^{\beta}}^0[N_n^2]}\geq c_6(\beta)^2/c_7,
\end{equation}
since $\E_{R_n^{\beta}}^0[N_n] = \E_{R_n^{\beta}}^0[N_n \mathbb{I}_{N_n>0}]$. We have thus reached the claim.

\bigskip

\noindent \textbf{Step 2: lower and upper bounds for general boundary conditions.} Using the ordering between boundary conditions Eq.(\ref{extremality}), the lower bound that we have just proved for free boundary conditions actually implies the lower bound for any boundary conditions $\xi$. 

For the upper bound, consider a rectangle $R$ with dimensions $n\times m$ with $m \in \llbracket \beta_1n,\beta_2n\rrbracket$ and with boundary conditions $\xi$. Using once again Eq.(\ref{extremality}), it is sufficient to address the case of wired boundary conditions, and in this case, the probability that there exists a dual crossing from the left side to the right side is at least $c_-=c_-(1/\be_2,1/\beta_1)$, since the dual model has free boundary conditions. We deduce, using the self-duality property, that
\begin{equation}
\P_{R}^{\xi}(\C_v(R))\leq 1-\P_{R}^1(\C_h^*(R))=1-\P_{R^*}^0(\C_h(R^*))\leq 1-c_-,
\end{equation}
where we use the notation $\C_h^*$ for the existence of a horizontal dual crossing, and $R^*$ is as usual the dual graph of $R$ (note that we have implicitly used the invariance by $\pi/2$-rotations). This concludes the proof of Theorem \ref{RSW}.

\end{proof}

\section{Consequences for the FK Ising and the (spin) Ising models} \label{csq}

\subsection{RSW-type crossing bound for the Ising model}

Theorem \ref{RSW} can also be applied to the Ising model, using the Edwards-Sokal coupling. However, we have to be a little careful since not all boundary conditions can ``go through this coupling''.

\begin{corollary}
Consider the Ising model with $(+)$ or free boundary conditions in a rectangle $R$ with dimensions $n$ and $m < \beta n$. There exists a constant $c_{\beta}>0$ such that
$$\mathbb{P}_R^{\textrm{free} / +}(\mathcal{C}_v^+(R)) \geq c_{\beta},$$
where $\mathcal{C}_v^+$ denotes the existence of a vertical $(+)$ crossing.
\end{corollary}

We could state this result for more general boundary conditions, for instance $(+)$ on one arc and free on the other arc. The corresponding result for $(-)$ boundary conditions is actually not expected to hold: one can notice for example that in any given smooth domain, a CLE(3) process -- the object describing the scaling limit of cluster interfaces -- a.s. does not touch the boundary.

\subsection{Power-law decay of the magnetization at criticality}

We start by stating an easy consequence of Theorem \ref{RSW}. We consider the box $S_n=\llbracket-n,n\rrbracket^2$, its boundary being denoted as usual by $\partial S_n$. We also introduce $S_{m,n}$ the annulus $S_n \setminus \mathring{S}_m$ of radius $m <n$ centered on the origin, and we denote by $\mathcal{C}(S_{m,n})$ the event that there exists an open circuit surrounding $S_m$ in this annulus.

\begin{corollary}[circuits in annuli] \label{circuits}
For every $\beta < 1$, there exists a constant $c_{\beta}>0$ such that for all $n$ and $m$, with $m \leq \beta n$,
$$\mathbb{P}_{S_{m,n}}^0(\mathcal{C}(S_{m,n})) \geq c_{\beta}.$$
\end{corollary}

\begin{proof}
This follows from Theorem \ref{RSW} applied in the four rectangles $R_B= \llbracket -n,n \rrbracket \times \llbracket -n,-m \rrbracket$, $R_L= \llbracket -n,-m \rrbracket \times \llbracket -n,n \rrbracket$, $R_T= \llbracket -n,n \rrbracket \times \llbracket m,n \rrbracket$ and $R_R= \llbracket m,n \rrbracket \times \llbracket -n,n \rrbracket$. Indeed, if there exists a crossing in each of these rectangles in the ``hard'' direction, one can construct from them a circuit in $S_{m,n}$.

Now, consider any of these rectangles, $R_B$ for instance. Its aspect ratio is bounded by $2/(1-\beta)$, so that Theorem \ref{RSW} implies that there is a horizontal crossing with probability at least
$$\mathbb{P}_{R_B}^0(\mathcal{C}_H(R_B)) \geq c > 0.$$
Combined with the FKG inequality, this allows us to conclude: the desired probability is at least $c_{\beta} = c^4 > 0$.
\end{proof}

\begin{proposition}[power-law decay for the magnetization]\label{zero-magnetization}
For $p= p_{\mathrm{sd}}$, there exists a unique infinite-volume measure $\P$. For this measure, there is almost surely no infinite open cluster. Moreover, there exist constants $\alpha,c>0$ such that for all $n \geq 0$,
\begin{equation}\label{up}
\P(0 \leadsto \partial S_n)\leq \frac{c}{n^{\alpha}}.
\end{equation}
\end{proposition}

This result implies in particular that $\P(0 \leadsto \infty) = 0$, in other words that there is no magnetization at $p= p_{\mathrm{sd}}$. This result also applies to the Ising model: the magnetization at the origin decays at least as a power law.

\begin{remark}
We would like to mention that an alternative proof of the fact that there is no spontaneous magnetization at criticality can be found in \cite{W_percolation,Gr2}. Also, we actually know from Onsager's work \cite{KaufmanOnsager} that the connection probability follows a power law as $n \to \infty$, described by the one-arm plane exponent $\alpha_1 = 1/8$. It should be possible to prove the existence and the value of this exponent using conformal invariance, as well as the arm exponents for a larger number of arms. More precisely, one would need to consider the probability of crossing an annulus a certain (fixed) number of times in the scaling limit, and analyze the asymptotic behavior of this probability as the modulus tends to $\infty$. Theorem \ref{RSW} then implies the so-called quasi-multiplicativity property, which allows one to deduce, using concentric annuli, the existence and the value of the arm exponents for the discrete model.
\end{remark}

\begin{proof}
We first note that it is classical that the non-existence of infinite clusters implies uniqueness of the infinite-volume measure: it is thus sufficient to prove Equation Eq.(\ref{up}). We consider the annuli $A_n=S_{2^n,2^{n+1}}$ for $n \geq 1$, and $\mathcal{C}^*(A_n)$ the event that there is a dual circuit in $A_n^*$. We know from Corollary \ref{circuits} that there exists a constant $c > 0$ such that
\begin{equation}
\P_{A_n}^1(\mathcal{C}^*(A_n)) \geq c
\end{equation}
for all $n \geq 1$. By successive conditionings, we then obtain
\begin{equation}
\P(0 \leadsto \partial S_{2^{N}}) \leq \prod_{n=0}^{N-1}\P_{A_n}^1((\mathcal{C}^*(A_n))^c) \leq (1-c)^N,
\end{equation}
and the desired result follows.
\end{proof}

\begin{remark}
Note that together with sharp threshold arguments developed by Graham and Grimmett \cite{Grimmett2}, these crossing estimates also provide a geometric proof that the critical point is $p_c= p_{\mathrm{sd}}(2)=\sqrt{2}/(1+\sqrt{2})$ (which then gives the critical temperature of the Ising model, thanks to the Edwards-Sokal coupling).
\end{remark}

\subsection{Regularity of interfaces and tightness}

Theorem \ref{RSW} can be used to apply the technology developed by Aizenman and Burchard \cite{AB}, to prove regularity of the collection of interfaces, which implies tightness using a variant of the Arzel\`a-Ascoli theorem. 

This compactness property for the set of interfaces is important to construct the scaling limits of discrete interfaces, once we have a way to identify their limit uniquely (using for instance the so-called martingale technique, detailed in \cite{Sm2}). Here, the fermionic observable provides a conformally invariant martingale, and its convergence to a holomorphic function has been proved in \cite{Sm3}, leading to the following important theorem:
\begin{theorem}[Smirnov \cite{Sm4}] \label{convergence to SLE(16/3)}
For any Dobrushin domain $(\D,a,b)$, with discrete approximations $(\D_\e,a_{\e},b_{\e})$, the $\P_{\D_\e,a_{\e},b_{\e}}$-law of the exploration path $\g_{\e}$ from $a_{\e}$ to $b_{\e}$ converges weakly to the law of a chordal \emph{SLE(16/3)} path in $\D$, from $a$ to $b$. 
\end{theorem}

We briefly explain how one can use the crossing bounds to obtain the compactness of the interfaces. Note that this result has also been proved, in a different way, in \cite{K} and in the forthcoming article \cite{KS1}.

As usual, curves are defined as continuous functions from $[0,1]$ into a bounded domain $\mathcal{D}$ -- more precisely, as equivalence classes up to strictly increasing reparametrization. The \emph{curve distance} is then just the distance induced by the norm of uniform convergence.

Let $A(x;r,R)$ be the annulus $\B(x,R)\setminus \B(x,r)$. We denote by $\A_k(x;r,R)$ the event that there are $2k$ crossings of the curve from $\B(x,r)$ to $\B(x,R)^c$. 

\begin{theorem}[Aizenman-Burchard \cite{AB}]\label{Aizenman-Burchard}
Let $\D$ be a compact domain and denote by $\P_{\e}$ the law of a random curve $\tilde{\g}_{\e}$ with short-distance cut-off $\e>0$. If for any $k>0$, there exists $C_k<\infty$ and $\lambda_k>0$ such that for all $\e<r<R$ and $x\in \D$,
\begin{equation}
\P_{\e}(\A_k(x;r,R))\leq C_k \Big(\frac{r}{R}\Big)^{\lambda_k},
\end{equation}
and $\lambda_k\rightarrow \infty$, then the curves $(\tilde{\g}_{\e})$ are precompact for the weak convergence associated with the curve distance.
\end{theorem}

This theorem can be applied to the family $(\gamma_{\epsilon})$ of exploration paths defined in Theorem \ref{convergence to SLE(16/3)}, using the following argument. If $\A_k(x;r,R)$ holds, then there are $k$ open paths, alternating with $k$ dual paths, connecting the inner boundary of the annulus to its outer boundary. Moreover, one can decompose the annulus $A(x;r,R)$ into roughly $\log_2(R/r)$ annuli of the form $A(x;r,2r)$, so that it is actually sufficient to prove that
\begin{equation}
\P(\A_k(x;r,2r))\leq c^k
\end{equation}
for some constant $c<1$.  Since the paths are alternating, one can deduce that there are $k$ open paths, each one being surrounded by two dual paths. Hence, using successive conditionings and the Domain Markov property, the probability for each of them is smaller than the probability that there is a crossing in the annulus, which is less than some constant $c<1$ by Corollary \ref{circuits} (note that this reasoning also holds on the boundary).

Hence, Theorem \ref{Aizenman-Burchard} implies that the family $(\g_{\e})$ is precompact for the weak convergence.

\subsection{Half-plane one-arm exponent for the FK Ising model and boundary magnetization for the Ising model} \label{one_arm_exponent}

As a by-product of our proofs, in particular of the estimates of Section \ref{comp-harm-mes}, one can also obtain the value of the
critical exponent for the boundary magnetization in the Ising model, near a free boundary arc (assuming it is smooth), and the corresponding one-arm half-plane exponent for the FK Ising model.

Let us first consider the one-point magnetization $\E_{\D,a,b} [\sigma_x]$ 
for the Ising model at criticality in a discrete domain $(\D, a, b)$ with free boundary conditions on the
counterclockwise arc $(ab)$ and $(+)$ boundary conditions on the other arc $(ba)$.

\begin{proposition}
There exist positive constants $c_1$ and $c_2$ such that for any discrete 
domain $(\D, a, b)$ with $a = (- n ,0)$ and $b = (n ,0)$ ($n \geq 0$), containing
the rectangle $R_n = \llbracket -n,n \rrbracket \times \llbracket 0,n \rrbracket$ and such that its boundary contains the lower arc $\llbracket -n, n \rrbracket \times \{0\}$, 
we have
\begin{equation}
c_1 n^{-1/2} \leq \E_{\D, a, b} [ \sigma_0 ] \leq c_2 n^{-1/2},
\end{equation}
uniformly in $n$.
\end{proposition}

\begin{proof}
The magnetization at the origin can be expressed, in the corresponding FK representation, as the probability that the origin is connected to the wired counterclockwise arc $(ba)$. This probability can now be estimated exactly as in Section \ref{section_proof}: it is equal to the probability that the FK interface passes through $0$ (since the origin is on the free boundary of $\D$), which is itself the modulus of the fermionic observable on $(\D, a, b)$ evaluated at $0$. Now, it suffices to use Proposition \ref{uniform comparability} to compare this to the two harmonic measures, and then estimates similar to the estimates in Lemmas \ref{lower_bound} and \ref{upper_bound}.
\end{proof}

This result can be equivalently stated for the one-arm half-plane probability for FK percolation:

\begin{proposition}
Consider the rectangle $R_n = \llbracket -n,n \rrbracket \times \llbracket 0,n \rrbracket$. There exist
positive constants $c_1$ and $c_2$ such that for any boundary conditions $\xi$ such that
the lower arc $\partial^- B_n$ is free, one has
\begin{equation}
	c_1 n^{-1/2} \leq \P_{R_{n}}^{\xi} ( 0 \leadsto \partial^+ R_n)
	\leq c_2 n^{-1/2},
\end{equation}
uniformly over all $n$.
\end{proposition}

\begin{proof}
We get the upper bound using monotonicity and the previous proposition, since $(+)$ boundary conditions in
the Ising model correspond to wired boundary conditions in the corresponding FK representation. 
For the lower bound, 
by Theorem \ref{RSW} and the FKG inequality, we
can enforce the existence of a crossing in the half-annulus $R_n\setminus R_{n/2}$ 
that disconnects 0 from $\partial R_n\setminus \partial^-R_n$ to the price of a constant independent of $\xi$.
Using monotonicity and FKG, the probability that $0$ is connected by an open path to this crossing 
(conditionally on its existence) is larger than the probability that $0$ is connected to the boundary with wired
boundary conditions, without conditioning. Hence, the lower bound of the previous proposition gives the desired result.
\end{proof}

\begin{remark}
Note that contrary to the power laws established using \emph{SLE}, there are no potential logarithmic corrections here -- as is the case with the ``universal'' arm exponents for percolation (corresponding to 2 and 3 arms in the half-plane, and 5 arms in the plane). Furthermore, one can follow the same standard reasoning as for percolation, based on the RSW lower bound, to prove that the two- and three-arm half-plane exponents, with alternating ``types'' (primal or dual), have values $1$ and $2$ respectively.
\end{remark}

\subsection{$n$-point functions for the FK Ising and the Ising models}

Since the work of Onsager \cite{Onsager}, it is well-known that for the Ising model at criticality, the magnetization at the middle of a square of side length $2m$ with $(+)$ boundary conditions decays like $m^{-1/8}$. 
It is then tempting to say that the correlation of two spins at distance $m$ in the plane
(in the infinite-volume limit, say) decays like $m^{-1/4}$, and this is indeed what happens. To the knowledge of
the authors, there is no straightforward generalization of Onsager's work that allows to derive this without difficult computations. However, this result can be made rigorous very easily with the help of Theorem \ref{RSW}. We give here only a result for two-point correlation functions, but exponents for $n$-spin correlations, for instance, can be obtained using exactly the same method.

Let us first interpret Onsager's result in terms of the FK representation. 
\begin{lemma}
Let $B_m$ be the square $\llbracket -m,m \rrbracket^2$ with arbitrary boundary conditions $\xi$. Then there
exist two constants $c_1$ and $c_2$ (independent of $m$ and $\xi$) such that we have 
\begin{equation}
c_1 m^{-1/8} \leq \P^{\xi} (0 \leadsto \partial B_m) \leq c_2 m^{-1/8}.
\end{equation}
\end{lemma}
\begin{proof}
Onsager's result gives the result with wired boundary conditions (since it is derived in terms
of the Ising model with $(+)$ boundary conditions), so by monotonicity it gives the upper bound. 
Using Theorem \ref{RSW}, we can obtain a lower bound
independent of the boundary conditions by enforcing the existence of a circuit in the annulus
$B_m \setminus B_{m/2}$, and using the FKG inequality. For that, we just need to lower the constant, using monotonicity:
the connection probability conditionally on the fact that there is a wired annulus around the origin is larger than the connection probability with $(+)$ boundary conditions on $\partial B_m$.
\end{proof}

We can now give the result for two-point correlation functions in the infinite-volume Ising model.
\begin{proposition}
Consider the Ising model on $\Z^2$ at critical temperature. There exist two positive constants $C_1$ and $C_2$ such that we have
\begin{equation}
	C_1 |x - y|^{-1/4} \leq \E [ \sigma_x \sigma_y ] \leq C_2 |x - y|^{-1/4},
\end{equation}
where for any $x, y \in \Z^2$, we denote by $\sigma_x$ and $\sigma_y$ the spins at $x$ and $y$.
\end{proposition}
\begin{proof}

The 2-spin correlation $\E [ \sigma_x \sigma_y ]$ can be expressed, in the corresponding FK representation, as the probability
of the event $\{ x \leadsto y \}$. Let now $m$ be the integer part of $|x-y|/4$.
The upper bound is easy and does not rely on Theorem \ref{RSW}: the event that
$x$ is connected to $y$ implies that $x$ is connected to $x + \partial B_m$ and that $y$ is connected
to $y + \partial B_m$. Using the Domain Markov property, these two events are independent conditionally
on the states on the boundaries of the boxes. Using the previous lemma, we get the upper bound.

Let us turn now to the lower bound. We can enforce the existence of a connected ``$8$'' in the discrete domain 
$$\left[ (x + B_{2m + 2})\cup (y + B_{2m + 2}) \right] \setminus \left[ (x + B_{m})\cup (y + B_{m}) \right]$$
that surrounds both $x$ and $y$ and separates them: this costs only a positive constant $\alpha$, independent of $m$, using Theorem \ref{RSW} in well-chosen rectangles and the FKG inequality. Using once again the FKG inequality, we get that
\begin{equation}
\P ( x \leadsto y ) \geq \alpha \P (x \leadsto x + \partial B_{2m+2}) \cdot \P ( y \leadsto y + \partial B_{2m+2} ),
\end{equation}
and combined with the previous lemma, this yields the desired result.
\end{proof}

\section*{Acknowledgments}

This research was initiated during a semester spent by P.N. at Universit\'e de Gen\`eve, and P.N. would like to thank the mathematics department there for its hospitality, in particular Stanislav Smirnov. P.N.'s visit was made possible by the NSF grant OISE-07-30136. H.D.-C.'s research was supported  by the Marie-Curie grant S16472. This research was also supported in part by the Swiss FNS. H.D.-C. and C.H. are also particularly thankful to Stanislav Smirnov for his constant support during their PhD. Finally, the three authors enjoyed fruitful and stimulating discussions with many people, and they are particularly grateful to Vincent Beffara, Dimitri Chelkak, Geoffrey Grimmett, Charles Newman, Stanislav Smirnov, Yvan Velenik and Wendelin Werner.

\footnotesize

\begin{flushright}
\textsc{D\'epartement de Math\'ematiques}\\
\textsc{Universit\'e de Gen\`eve}\\
\textsc{Gen\`eve, Suisse}\\
\textsc{E-mail: }\texttt{clement.hongler@unige.ch}, \texttt{ hugo.duminil@unige.ch}

\end{flushright}

\medbreak

\begin{flushright}

\textsc{Courant Institute of Mathematical Sciences}\\
\textsc{New York University}\\
\textsc{New York, U.S.A.}\\
\textsc{E-mail: }\texttt{nolin@cims.nyu.edu} 

\end{flushright}

\end{document}